\documentclass{amsart}

\usepackage{epsfig}

\newtheorem{theorem}{Theorem}[section]
\newtheorem{proposition}[theorem]{Proposition}
\newtheorem{definition}[theorem]{Definition}
\newtheorem{lemma}[theorem]{Lemma}
\newtheorem{corollary}[theorem]{Corollary}
\newtheorem{remark}[theorem]{Remark}
\newtheorem{example}[theorem]{Example}
\newtheorem{hypothesis}[theorem]{Hypothesis}
\newtheorem*{acknowledgement}{Acknowledgement}

\newcommand{\iot}{\iota(\theta)}
\newcommand{\R}{\mathbb{R}}
\newcommand{\QQ}{\mathcal{Q}}
\newcommand{\PP}{\mathcal{P}}
\newcommand{\C}{\mathbb{C}}
\newcommand{\G}{\mathcal{G}}
\newcommand{\Z}{\mathbb{Z}}
\newcommand{\D}{\mathcal{D}}
\newcommand{\DR}{D_{\mathrm{R}}}
\newcommand{\twist}{\mathcal{I}(\theta)}
\newcommand{\V}{\mathcal{V}}
\newcommand{\M}{\mathcal{M}}
\newcommand{\g}{\mathfrak{g}}
\newcommand{\id}{\mathrm{id}}
\newcommand{\F}{\mathcal{S}_{2n}}
\newcommand{\OO}{\mathcal{O}}
\newcommand{\COO}[1]{\overline{\mathcal{O}_{#1}}}
\newcommand{\ontop}[2]{\genfrac{}{}{0pt}{}{#1}{#2}}
\newcommand{\BG}{\mathrm{BG}}
\newcommand{\br}{\mathrm{Br}}
\newcommand{\bg}{\BG}
\newcommand{\fix}{F_{2n}}
\newcommand{\out}{\mathrm{out}}
\newcommand{\supp}{\mathrm{supp}}
\newcommand{\SL}{\mathrm{SL}}
\newcommand{\Sp}{\mathrm{Sp}}

\begin{document}
\author{Axel Hultman}
\address{Department of Mathematics, KTH-Royal Institute of Technology,
  SE-100 44, Stockholm, Sweden.}
\email{axel@math.kth.se}

\title[Criteria for rational smoothness]{Criteria for rational smoothness of some symmetric orbit closures}

\begin{abstract}
Let $G$ be a connected reductive linear algebraic group over $\C$ with
an involution $\theta$. Denote by $K$ the subgroup of fixed points. In
certain cases, the $K$-orbits in the flag variety $G/B$ are indexed by
the twisted identities $\iot = \{\theta(w^{-1})w\mid w\in W\}$ in the
Weyl group $W$. Under this assumption, we establish a criterion for
rational smoothness of orbit closures which generalises classical
results of Carrell and Peterson for Schubert varieties. That is, whether an
orbit closure is rationally smooth at a given point can be determined
by examining the degrees in a ``Bruhat
graph'' whose vertices form a subset of $\iot$. Moreover, an orbit
closure is rationally smooth everywhere if and only if its
corresponding interval in the Bruhat order on $\iot$ is rank symmetric.

In the special case $K=\Sp_{2n}(\C)$, $G=\SL_{2n}(\C)$, we strengthen
our criterion by showing that only the degree
of a single vertex, the ``bottom one'', needs to be examined. This
generalises a result of Deodhar for type $A$ Schubert varieties.

\end{abstract}

\maketitle

\section{Introduction}
Let $G$ be a connected reductive complex linear algebraic group equipped with
an automorphism $\theta$ of order $2$. There is a
$\theta$-stable Borel subgroup $B$ which contains a $\theta$-stable
maximal torus $T$ \cite[\S 7]{steinberg} with normaliser $N$. Let
$K=G^\theta$ be the fixed point subgroup. We may 
always assume $\theta$ to be the complexification of the Cartan
involution of some real form $G_\R$ of $G$.


The flag variety $X=G/B$ decomposes into finitely many orbits under the
action of the symmetric subgroup $K$ by left translations. A
natural ``Bruhat-like'' partial order on the set of orbits $K\backslash
X$ is defined by 
inclusion of their closures. Let $V$ denote this poset. Richardson
and Springer \cite{RS, RS2} defined a
poset map $\varphi: V\to \br(W)$, where $\br(W)$ is the Bruhat order
on the Weyl group $W=N/T$. The image of $\varphi$ is contained in the
set of {\em twisted involutions} $\twist=\{w\in W\mid
\theta(w)=w^{-1}\}$. In general, $\varphi$ is neither injective nor
surjective. For certain choices of $G$ and $\theta$, however,
$\varphi$ produces a poset isomorphism $V\cong \br(\iota(\theta))$, where
$\iota(\theta) = \{\theta(w^{-1})w\mid w\in W\}\subseteq \twist$
is the set of {\em twisted identities} and $\br(\cdot)$ denotes
induced subposet of $\br(W)$. In Section \ref{se:restriction}, we shall make
explicit under what circumstances this fairly restrictive
assumption holds. Now suppose that $\varphi$ is such an isomorphism and
let $\COO{w}$, $w\in \iot$, denote the closure of the orbit
$\OO_w=\varphi^{-1}(w)$. In this article we express the
rationally singular locus of the symmetric variety $\COO{w}$ in terms
of the combinatorics of $\iot$.

With each $w\in \iot$, we shall associate a {\em Bruhat graph}
$\bg(w)$ with vertex 
set $I_w = \{u\in \iot\mid u\le w\}$. Our first main result, Theorem
\ref{th:regtoratsmooth}, states that $\COO{w}$ is rationally smooth at
$\OO_u$ if and only if $v$ is contained in $\rho(w)$ edges for all
$u\le v \le w$, where $\rho(w)$ is the rank of $w$ in $\br(\iot)$. In particular, $\COO{w}$ is rationally smooth if and only if
$\bg(w)$ is $\rho(w)$-regular. This latter statement also turns out to
be equivalent to the principal order ideal $\br(I_w)$ being
rank-symmetric; see Theorem \ref{th:rank_symmetry} below.

The assertions just stated generalise celebrated criteria due to Carrell
and Peterson \cite{CP} for rational smoothness 
of Schubert varieties. We recover their results in the
special case where $G=G^\prime\times G^\prime$ and $\theta(x,y)=(y,x)$.

The main brushstrokes of our proofs are completely similar to
those of Carrell and Peterson. Below the surface, however, their results
rely on delicate connections between Kazhdan-Lusztig polynomials and
the combinatorics of (ordinary) Bruhat graphs. Our chief contribution is to
extend these properties to a more general setting. Very roughly, here is
what we do:

First, properties of $\iot$ are established that combined with
results of Brion \cite{brion} imply a bound on the degrees in
$\bg(w)$ that generalises ``Deodhar's inequality'' for degrees in ordinary
Bruhat graphs of Weyl groups. 

Second, an explicit procedure, in terms of the combinatorics of $\iot$, for
computing the  ``$R$-polynomials'' of \cite{LV, vogan} is extracted from the
correspondence $V \leftrightarrow \iot$. Using this procedure we
establish several properties of these polynomials (and therefore
of Kazhdan-Lusztig-Vogan polynomials) and relate them to degrees in
the graphs $\bg(w)$. This generalises well known properties of ordinary
Kazhdan-Lusztig polynomials and $R$-polynomials and how they are
related to ordinary Bruhat graphs.

The most prominent example where our results say something which is
not contained in \cite{CP} is
$G=\SL_{2n}(\C)$, $K=\Sp_{2n}(\C)$. For this setting, we prove the
stronger statement (Corollary \ref{co:bottom}) that the degree of the
bottom vertex alone suffices 
to decide rational smoothness. That is, $\COO{w}$ is rationally smooth
at $\OO_u$ if and only if the degree of $u$ in $\bg(w)$ is
$\rho(w)$. This is analogous to a corresponding result for type $A$
Schubert varieties which is due to Deodhar \cite{deodhar}. Again, that
result is contained in ours as a special case. 

\begin{remark}{\em 
After a preliminary version of this article was circulated,
McGovern \cite{mcgovern} has applied our results in 
order to deduce a criterion for (rational) smoothness in the case $G=
\SL_{2n}(\C)$, $K=\Sp_{2n}(\C)$ in terms of pattern avoidance
among fixed point free involutions. Moreover, he proved that
in this case the rationally singular loci in fact coincide with the
singular loci. 
}\end{remark}

Closures of symmetric orbits are interesting objects in their own right, but
another important reason to study their singularities is their impact
on representation theory. We outline this connection while describing
one of our main tools, Kazhdan-Lusztig-Vogan polynomials, in the
next section.  

In Section \ref{se:restriction}, we make precise the assumptions on
$\theta$ for which our results are valid. Thereafter, the Bruhat graphs $\bg(w)$ are introduced in
Section \ref{se:BG}. Our Carrell-Peterson type criteria for rational
smoothness are deduced in Section \ref{se:main}. Finally, in Section
\ref{se:bottom}, we prove that the bottom vertex alone suffices to decide
rational smoothness when $G= \SL_{2n}(\C)$, $K=\Sp_{2n}(\C)$.

\begin{acknowledgement}{\em 
The author is grateful to W.\ M.\ McGovern for many helpful discussions.
}\end{acknowledgement}

\section{KLV polynomials and representation theory}\label{se:KLV}
In the present  paper, the principal method for detecting rational
singularities of symmetric orbit closures is via Kazhdan-Lusztig-Vogan
polynomials. Here, we 
briefly review some of their properties and establish notation. For
more information we 
refer the reader to \cite{LV} or \cite{vogan}. Our terminology chiefly
follows the latter reference. 

Let $\D$ denote the set of pairs $(\OO, \gamma)$, where $\OO\in
K\backslash X$ and $\gamma$ is a $K$-equivariant local system on
$\OO$. The choice of $\gamma$ is equivalent to the choice of a
character of the component group of the stabiliser $K_x$ of a
point $x\in \OO$. In 
particular, $\gamma$ is unique if $K_x$ is connected. Since $\OO$ is
determined by $\gamma$, we may abuse notation and write $\gamma$ for
$(\OO,\gamma)$. With each pair $\gamma, \delta\in \D$, we associate
polynomials $R_{\gamma, \delta}, P_{\gamma, \delta}\in \Z[q]$. The
$R$-polynomials can be computed using  
a recursive procedure which we refrain from stating in full generality
here; see \cite[Lemma 6.8]{vogan} for details. A special case
sufficient for our purposes is formulated in Proposition \ref{pr:R} below. 

Let $\M$ denote the free $\Z[q,q^{-1}]$ module with basis $\D$. For
fixed $\delta\in \D$, we have in $\M$ the identity
\[
 q^{-l(\delta)} \sum_{\gamma\le \delta}P_{\gamma,\delta}(q)\gamma =
 \sum_{\beta\le \gamma\le \delta}(-1)^{l(\beta)-l(\gamma)}
 q^{-l(\gamma)}P_{\gamma,\delta}(q^{-1})R_{\beta,\gamma}(q)\beta
\]
which subject to the restrictions $P_{\gamma,\gamma}=1$ and
$\deg(P_{\gamma,\delta})\le (l(\delta)-l(\gamma)-1)/2$ uniquely
determines the {\em Kazhdan-Lusztig-Vogan (KLV)} polynomials
$P_{\gamma,\delta}$ \cite[Corollary 6.12]{vogan}.\footnote{Note that there is a typo which has an impact on the cited result. We are
  grateful to D.\ A.\ Vogan for pointing out that the displayed
  formula in the statement of \cite[Lemma 6.8]{vogan} should read 
\[
D(\delta) = u^{-l(\delta)}\sum_\gamma(-1)^{l(\gamma)-l(\delta)}R_{\gamma,\delta}(u)\gamma.
\]} Here, $l(\cdot)$
indicates the dimension of the corresponding orbit, and the order on
$\D$ is the {\em Bruhat $\G$-order} \cite[Definition 5.8]{vogan}.

KLV polynomials serve as measures of the singularities of symmetric
orbit closures; cf.\ \cite[Theorem 1.12]{vogan}. In particular, their
coefficients are nonnegative. Another consequence is the
following:
\begin{proposition}
Let $\le$ denote the order relation in $V$, i.e.\ containment among
orbit closures. Given orbits $\PP,\OO \in K\backslash X$ with $\PP \le
\OO$, let $\delta=(\OO,\C_\OO)$, where $\C_\OO$
is the trivial local system. Then, $\COO{}$ is rationally smooth at some
(equivalently, every) point in $\PP$ if and only if 
\[
P_{\gamma, \delta} =
\begin{cases}
1 & \text{if $L=\C_\QQ$,}\\
0 & \text{if $L\neq \C_\QQ$,}
\end{cases}
\]
for all $\gamma=(\QQ,L)\in \D$ with $\PP\le \QQ\le \OO$.
\end{proposition}

The gadgets just described are fundamental ingredients in the
representation theory of $G_\R$. Fix an infinitesimal character
for $G_\R$. Then, $\D$ is in bijective correspondence with two families
of $G_\R$-representations with this infinitesimal character. Given
$\gamma\in \D$, there is the standard $(\g, K_\R)$-module $X(\gamma)$
induced from a discrete series representation, and there is the
irreducible $(\g, K_\R)$-module $\overline{X}(\gamma)$. The transition
between the two families is governed by the KLV polynomials. Namely,
one has
\[
\overline{\Theta}(\gamma) =
\sum_{\gamma^\prime}(-1)^{l(\gamma)-l(\gamma^\prime)}P_{\gamma^\prime,\gamma}(1)\Theta(\gamma^\prime), 
\]
where $\Theta(\gamma)$ and $\overline{\Theta}(\gamma)$ denote the
characters of $X(\gamma)$ and $\overline{X}(\gamma)$, respectively
\cite{LV, vogan}.

\section{Restricting the involution}\label{se:restriction}
Consider the set $\V = \{g\in G\mid \theta(g^{-1})g\in N\}$. The set
of orbits $K\backslash \V/T$ parametrises $K\backslash X$. In this
way, the map
$\V\to W$ given by $g\mapsto \theta(g^{-1})gT$ induces the map
$\varphi:V\to W$ which was mentioned in the introduction. Observe that
the image of $\varphi$ is contained in $\twist$. 

Throughout
this paper we shall only allow certain choices of $\theta$. More
precisely, {\em we from now on assume that $\theta$ obeys the following
condition:}
\begin{hypothesis}\label{hy:theta}
The fixed point subgroup $K$ is connected. Moreover, $\varphi:V\to W$
satisfies $\varphi(v_0)\in \iot$, where 
$v_0\in V$ is the maximum element, i.e.\ the dense orbit.
\end{hypothesis}

\begin{remark} {\em 
If $G$ is semisimple and simply connected, then $K$ is necessarily
connected. This result is due to Steinberg \cite[Theorem
8.1]{steinberg}. In some sense, the general situation can be reduced
to the study of semisimple simply connected $G$; see \cite{RS}.}
\end{remark}

Several consequences are collected in the next proposition. We let
$\Phi$ denote the root system of $G,T$ and write $R\subset W$ for the
corresponding set of reflections.

\begin{proposition}\label{pr:theta}
Hypothesis \ref{hy:theta} implies the following:
\begin{itemize}
\item[(i)] The map $\varphi$ yields a poset isomorphism $V\to
  \br(\iot)$.

\item[(ii)] There is a unique $K$-equivariant local system, namely
  $\C_\OO$, on each orbit $\OO\in 
  K\backslash X$. In particular, the sets $\D$, $K\backslash X$
  and $\iot$ may be identified, and the Bruhat $\G$-order
  on $\D$ coincides with $V$ and $\br(\iot)$.

\item[(iii)] Let $\alpha\in \Phi$ and denote by $G_\alpha\subseteq G$ the
corresponding rank one semisimple group. Then, we are in one of the
following two situations:
\begin{itemize}
\item[(a)] The root $\alpha$ is {\em compact imaginary}. That is, $G_\alpha
  \subseteq K$.
\item[(b)] The root $\alpha$ is {\em complex} (meaning
  $\theta(\alpha)\neq \alpha$) and  $\theta(\alpha)+\alpha\not \in \Phi$.
\end{itemize}

\item[(iv)] If $r\in R$, then $\theta(r)r = r\theta(r)$.
\item[(v)] The poset $\br(\iot)$ is graded with rank function $\rho$
  being half the ordinary Coxeter length. Moreover, $\rho(w) =
  l(\OO_w)-l(\OO_\id)$.
\end{itemize}
\begin{proof}
Assertion (i) follows from Richardson and Springer's \cite[Proposition
9.16]{RS}. 

For (ii), the local system on $Kx$, $x\in X$, is unique if the
isotropy subgroup $K_x$ is connected. Under Hypothesis
\ref{hy:theta}, Springer's \cite[Proposition 4.8]{springer1} implies that $K_x$
is connected if the torus fixed point group $T^\theta$ is
connected. Since $K$ is connected, this follows from \cite[Lemma
  5.1]{richardson}.

In order to prove (iii), suppose $\theta(\alpha)=\alpha$ but $G_\alpha
\not \subseteq K$. Then, the corresponding reflection $r_\alpha\in R$
is in the image of $\varphi$ by \cite[Lemma 2.5(i)]{springer2}. This
image is, however, $\iot$ which does not contain any reflections.

If $\theta(\alpha) \neq \alpha$ and $\beta=\alpha+\theta(\alpha)\in
\Phi$, then $\theta(\beta)=\beta$ and $G_\beta\not \subseteq K$ by
\cite[Lemma 2.6]{springer1}. This once again leads to the above
contradiction.

Concerning (iv), assuming $\theta(r)\neq r$ \cite[Lemma
  2.5]{springer1} implies that the 
dihedral group generated by $r$ and $\theta(r)$ is either of type
$A_1\times A_1$ or of type $A_2$. If the latter were true, we would have
$\theta(\alpha)+\alpha \in \Phi$, where $\alpha$ is the positive root
corresponding to $r$. This contradicts part (iii), and the claim is
established. 

Finally, the first part of (v) follows from (iv) in conjunction with
\cite[Theorem 4.6]{hultman}. The second is then immediate from
\cite[Theorem 4.6]{RS}.
\end{proof}
\end{proposition}

The following example allows us to consider many of our results as
generalisations of statements about Schubert varieties.

\begin{example}\label{ex:Schubert}
If $G^\prime$ is a connected reductive complex linear algebraic
group and $G = G^\prime\times G^\prime$, the involution $\theta$ which
interchanges the two factors makes $K$ the diagonal subgroup. In this
case, $\iot = \twist$, so Hypothesis \ref{hy:theta} is satisfied. The
poset $\br(\iot$ coincides with $\br(W^\prime)$, where $W^\prime$ is the Weyl group
of $G^\prime$. There is a one-to-one correspondence between
$K$-orbits in $X$ and Schubert cells in the Bruhat decomposition of
the flag variety of $G^\prime$ which preserves a lot of structure
including the property of having rationally smooth closure at a
given orbit.  
\end{example}

In addition to the setting in Example \ref{ex:Schubert} there are a
few more cases that satisfy Hypothesis \ref{hy:theta}. They are
denoted $A\,II$,  
$D\,II$ and $E\,IV$ in the classification of symmetric spaces
$G_\R/K_\R$ given e.g.\ in
Helgason \cite{helgason}.\footnote{The ``usual'' construction of
  $D\,II$ would yield $G=\mathrm{SO}_{2n}(\C)$,
  $K=\mathrm{S}(\mathrm{O}_{2n-1}(\C)\times \mathrm{O}_1(\C))\cong
  \mathrm{O}_{2n-1}(\C)$ so that $K$ is disconnected. However, passing to the fundamental
  cover, we have $G=\mathrm{Spin}_{2n}(\C)$, $K=\mathrm{Spin}_{2n-1}(\C)$ in
  agreement with Hypothesis \ref{hy:theta}.} The corresponding Weyl groups are $A_{2n+1}$,
$D_n$ and $E_6$, respectively, with $\theta$ in each case
restricting to the Weyl group as the
unique nontrivial Dynkin diagram involution. Types $D$ and $E$
could in principle be handled separately. In the former case,
$\iot$ has a very simple structure (cf.\ \cite[proof of Theorem
  5.2]{hultman}), whereas the latter admits a brute force
computation. Thus, the main substance lies in the
$A_{2n+1}$ case where $\br(\iot)$ is an incarnation of the
containments among closures of $\Sp_{2n}(\C)$ orbits in the flag
variety $\SL_{2n}(\C)/B$; see \cite[Example 10.4]{RS} for a discussion
of this case. Nevertheless, we have opted to keep our
arguments type independent regarding all assertions that
are valid in the full generality of Hypothesis \ref{hy:theta}. There
are two reasons. First, the natural habitat for Theorems
\ref{th:regtoratsmooth} and \ref{th:rank_symmetry} is the general
setting; no simplicity would be gained by formulating the arguments in
type $A$ specific terminology. Second, we hope that the less specialised
viewpoint shall prove suitable as point of departure for generalisations beyond
Hypothesis \ref{hy:theta}. 

\section{``Bruhat graphs''}\label{se:BG}
Let $*$ denote the $\theta$-twisted right conjugation action of $W$ on
itself, i.e.\ $u*w = \theta(w^{-1})uw$ for $u,w\in W$. Then $\iot$ is
the orbit of the identity element $\id \in W$. 

Recall that $I_w = \{u\in \iot\mid u\le w\}$.

\begin{definition}
Given $w\in \iot$, let $\BG(w)$ be the graph with vertex set $I_w$ and
an edge $\{u,v\}$ whenever $u=v*t\neq v$ for some reflection $t\in R$. 
\end{definition}

Notice that $\bg(u)$ is an induced subgraph of $\bg(w)$ if $u\le w$.
See Figure \ref{fi:bg} for an illustration. 

We shall refer to graphs of the form $\BG(w)$ as {\em Bruhat
  graphs}, because in the setting of Example \ref{ex:Schubert}, they coincide
  with (undirected versions of) the ordinary Bruhat graphs in
  $W^\prime$ introduced by Dyer \cite{dyer}.

\begin{figure}[htb]
\begin{center}
\epsfig{height=8cm, file=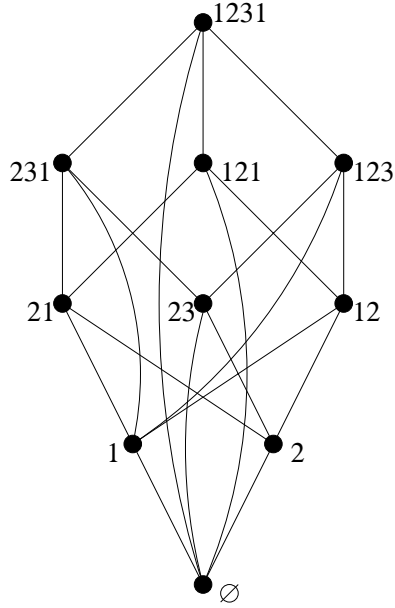}
\end{center}
\caption{A picture of the Bruhat graph $\bg(w)$ where
  $w=s_5s_3s_4s_5s_1s_2s_3s_1\in \iot\subset A_5$. Here, $s_i$ denotes
  the simple reflection $(i,i+1)$ in the usual manifestation of $A_5$ as
  the symmetric group $S_6$. The involution $\theta$ sends $s_{6-i}$
  to $s_i$. A vertex $u\in I_w$ is labelled by the indices of a sequence of
  simple reflections whose product $x$ satisfies
  $u=\theta(x^{-1})x$. The straight edges indicate the covering
  relation of $\br(\iot)$.}\label{fi:bg} 
\end{figure}

Our next goal is to show that (the first part of) Brion's \cite[Theorem
  2.5]{brion} implies lower bounds for the degrees in a Bruhat
graph. This essentially amounts to a reformulation of the
relevant parts of \cite{brion} using our terminology.

\begin{lemma}\label{le:edges}
Let $w\in \iot$ and $u,v\in I_w$, $u\neq v$. Write
$u=\theta(x^{-1})x$ for $x\in W$. The following are equivalent:
\begin{itemize}
\item[(i)] $\{u,v\}$ is an edge in $\BG(w)$.
\item[(ii)] There are exactly two distinct reflections $t\in R$ such
  that $u*t=v$.
\item[(iii)] There are exactly two distinct reflections $t\in R$ such
  that $\theta(x^{-1})\theta(t)tx=v$. If $t$ is one of these
  reflections, then $\theta(t)$ is the other. 
\end{itemize}
\begin{proof}

The implication (ii) $\Rightarrow$ (i) is obvious.

We have (iii) $\Rightarrow$ (ii), since
$\theta(x^{-1})\theta(t)tx=v$ if and only if
$u*r=v$ for $r=x^{-1}tx$. 

In order to show (i) $\Rightarrow$ (iii), assume $v =
\theta(x^{-1})\theta(r)rx = \theta(x^{-1})\theta(t)tx$, for $r,t\in
R$. In particular, $t\theta(t)=r\theta(r)$. Dyer's \cite[Lemma 3.1]{dyer}
shows that $\langle r,\theta(r),t,\theta(t)\rangle$ is a dihedral
reflection subgroup of $W$. Since $W$ is simply laced (which e.g.\ follows
from part (iv) of Proposition \ref{pr:theta} and inspection of finite
type Dynkin diagrams), this subgroup must be of
type $A_1\times A_1$; $A_2$ is not possible since
$t\theta(t)=\theta(t)t$. Hence, $\{r,\theta(r)\}=\{t,\theta(t)\}$, and
these two reflections are the possible candidates for $t$. 




\end{proof}
\end{lemma}

We are now in position to bound the degrees of a Bruhat
graph. Combining the
first part of Brion's \cite[Theorem 2.5]{brion} with part (iii) of Proposition
\ref{pr:theta} shows that the rank of a
vertex $v=\theta(x^{-1})x$ in $\BG(w)$ is at most half the number of
complex reflections (i.e.\ reflections that correspond to complex
roots) $t\in R$ such that $\theta(x^{-1})\theta(t)tx \leq
w$. By Lemma \ref{le:edges}, this is precisely the degree of $v$ in
$\BG(w)$. We thus have the following fact:

\begin{theorem} \label{th:inequality}
For $w\in \iot$, the degree of each vertex in $\BG(w)$ is at least $\rho(w)$. 
\end{theorem}

\begin{remark} {\em In the setting of Example \ref{ex:Schubert}, Theorem
\ref{th:inequality} specialises to ``Deodhar's inequality'' in $W^\prime$;
  see \cite[\S 6]{BL} and the references cited there.
}\end{remark} 

\begin{lemma}\label{le:comparability}
If $\{u,v\}$ is an edge in $\BG(w)$, then either $u < v$ or $v <
u$. Furthermore, $v$ has exactly $\rho(v)$ neighbours $u$ such that $u<v$.
\begin{proof}
Suppose $u=\theta(x^{-1})x \neq v =
u*t$ for some $x\in W$, $t\in R$. Define reflections
$r=u^{-1}\theta(t)u$ and $\tau = xtx^{-1}$. Using part (iv) of 
Proposition \ref{pr:theta}, we compute
\[
urt = v = \theta(x^{-1})\theta(\tau)\tau x = \theta(x^{-1})\tau
\theta(\tau) x = \theta(x^{-1})\tau \theta(x)ur = \theta(x^{-1})\tau
xr = utr.
\] 
Thus, $t$ and $r$ commute. Hence, $\{u,ut,\theta(t)u,v\} = u\langle
r,t \rangle$ and, by Dyer's \cite[Theorem 1.4]{dyer} the subgraph of
the (ordinary) Bruhat graph on $W$ induced by these four vertices is
isomorphic to the Bruhat
graph of the dihedral group on four elements. In particular, all pairs
$\{a,b\}\subset \{u,ut,\theta(t)u,v\}$ except at most one satisfy
$a\leq b$ or
$b\leq a$. Since the map $y\mapsto
\theta(y^{-1})$ is a poset automorphism of the Bruhat order which
sends $\theta(t)u$ to $ut$, $\{\theta(t)u,ut\}$ is the only
incomparable pair. This proves the first assertion.

For the second assertion, the above argument implies
\[
\{t\in R\mid vt < v\} = \{t \in R\mid v*t<v\}.
\]
It is well known that the left hand side has $\ell(v)=2\rho(v)$
elements. Lemma \ref{le:edges} concludes the proof.
\end{proof}
\end{lemma}

\section{A criterion for rational smoothness}\label{se:main}

In general, the recursion for the $R$-polynomials mentioned in Section
\ref{se:KLV} is technically
rather involved. Since we are assuming Hypothesis
\ref{hy:theta}, however, the situation is simpler. 
Proposition \ref{pr:theta} allows us to identify the indexing set $\D$ with
$\iot$. With $\DR(v)$ denoting the {\em descent set} of $v\in \iot$, i.e.\ the set of
simple reflections $s$ such that $vs<v$, or equivalently $v*s<v$, the
recursion takes the following explicit form:  

\begin{proposition}\label{pr:R}
Suppose $s\in \DR(v)$. Then, the $R$-polynomials satisfy 
\[
R_{u,v}(q) =
\begin{cases}
R_{u*s,v*s}(q) & \text{if $u*s<u$,}\\
qR_{u*s,v*s}(q)+(q-1)R_{u,v*s}(q) & \text{if $u*s>u$,}\\
-R_{u,v*s}(q) & \text{if $u*s = u$.}
\end{cases}
\]
\begin{proof}
Consider the free $\Z[q,q^{-1}]$ module $\M$ with basis
$\iot$. The definition of the map $T_s:\M \to \M$ formulated in 
\cite[Definition 6.4]{vogan} boils down to
\[
T_s w = 
\begin{cases}
qw & \text{if $w*s = w$,}\\
w*s & \text{if $w*s>w$,}\\
qw*s + (q-1)w & \text{if $w*s<w$,} 
\end{cases}
\]
for $w\in \iot$ (the relevant cases being (a), (b1) and (b2),
respectively). Equating coefficients in the identity
\[
\sum_{u \in \iot}(-1)^{\rho(u)}R_{u,w}(q)u = -\sum_{u\in
  \iot}(-1)^{\rho(u)}R_{u,w*s}(q)(T_s+1-q)u 
\]
(see \cite[proof of Lemma 6.8]{vogan}) now yields
\[
R_{u,w}(q) = 
\begin{cases}
R_{u*s,w*s}(q)-R_{u,w*s}(q)(q-1+1-q) & \text{if $u*s<u$,}\\
qR_{u*s,w*s}(q)-(1-q)R_{u,w*s}(q) & \text{if $u*s > u$,}\\
-R_{u,w*s}(q) & \text{if $u*s = u$,}\\
\end{cases}
\]
if $s\in S$ satisfies $w*s<w$.

\end{proof}
\end{proposition}

Together with the ``initial values'' $R_{u,u}(q) = 1$ and
$R_{u,v}(q)=0$ if $u \not \leq v$, we may calculate any $R_{u,v}$
using Proposition 
\ref{pr:R}. Rather than working with the actual $R$-polynomials, we shall
however find it more convenient to use the following simple variation:

\begin{definition}
For $u,v\in \iot$, let $Q_{u,v}(q) =
(-q)^{\rho(v)-\rho(u)}R_{u,v}(q^{-1})$.
\end{definition}
One readily verifies the following recursion:
\begin{proposition} \label{pr:Q-recursion}
For $s\in \DR(v)$, we have
\[
Q_{u,v}(q)=
\begin{cases}
Q_{u*s,v*s}(q) & \text{if $u*s<u$,}\\
qQ_{u*s,v*s}(q)+(q-1)Q_{u,v*s}(q) & \text{if $u*s >u$,}\\
qQ_{u,v*s}(q) & \text{if $u*s = u$.}
\end{cases}
\]
\end{proposition}
In particular, the $Q_{u,v}(q)$ are polynomials. In the setting of
Example \ref{ex:Schubert}, both the $R_{u,v}(q)$ and the $Q_{u,v}(q)$
coincide with the classical Kazhdan-Lusztig $R$-polynomials introduced
in \cite{KL}. The three lemmata coming up next hint that the
$Q_{u,v}(q)$ may provide the more useful generalisation. 

\begin{lemma} \label{le:edge}
For $u, v \in \iot$, we have
\[
Q^\prime_{u,v}(1)  =
\begin{cases}
1 & \text{if $u<v$ and $\{u,v\}$ is an edge in $\BG(v)$,}\\
0 & \text{otherwise.}
\end{cases}
\]
\begin{proof}
Suppose $s\in \DR(v)$. Differentiating
the equation in Proposition \ref{pr:Q-recursion} with respect to $q$,
and using that 
$Q_{u,v}(1)=R_{u,v}(1)=\delta_{u,v}$ (Kronecker's delta), it follows that
\[
Q_{u,v}^\prime(1) = Q^\prime_{u*s,v*s}(1)+\delta_{u,v*s}.
\]
It is clear that $\{u* s,v* s\}$ is an edge in $\BG(v)$ if and only if
the same is true about $\{u,v\}$. Employing induction on $\rho(v)$, it
thus suffices to 
show that $u* s<v* s$ if $v* s\neq u<v$ and $\{u,v\}$ is
an edge. Lemma \ref{le:comparability} shows that $u*s$ and $v*s$
are comparable in this situation. The assertion $u*s>v*s$ would
contradict the lifting property \cite[Lemma
  2.7]{hultman}, and we are done.
\end{proof} 

\begin{lemma}\label{le:Mobius}
Denote by $\mu$ the M\"obius function of $\br(\iot)$. Then, $\mu(u,v)
= Q_{u,v}(0)$ for all $u,v\in \iot$.
\begin{proof}
Let us induct on $\rho(v)$. The assertion
holds for $\rho(v)=0$ because $Q_{\id,\id}(q)=R_{\id,\id}(q)=1$. We shall
demonstrate that $\mu(u,v)$ satisfies the recursion
for $Q_{u,v}(0)$ derived from Proposition \ref{pr:Q-recursion}. 

Borrowing terminology from \cite{hultman}, call $[u,v]$ {\em full} if
every twisted involution in the interval $[u,v]$ is in fact a twisted
identity. Combining Philip Hall's theorem (see e.g.\ \cite[Proposition
3.8.5]{stanley}) with
\cite[Theorem 4.12]{hultman} shows that
\[
\mu(u,v) = 
\begin{cases}
(-1)^{\rho(v)-\rho(u)}& \text{if $[u,v]$ is full,}\\
0 & \text{otherwise.}
\end{cases}
\]
Pick $s\in \DR(v)$. In case $u* s = u$, $[u,v]$
is not full,
and $\mu(u,v)=0$ as 
desired. If $u* s > u$, it follows from \cite[Lemma 4.10]{hultman} that $[u,v* s]$ is full if and only if $[u,v]$ is
full. Thus, $\mu(u,v)=-\mu(u,v* s)$, and we are done. Finally,
suppose $u*s < u$. If $[u*s,v*s]$ is full then $[u,v]$ is also full, again by \cite[Lemma 4.10]{hultman}. On the other hand, \cite[Theorem 4.9]{hultman} implies
  that $\mu(u*s,v)=-\mu(u,v)$, so if $[u*s,v*s]$ (and
  therefore $[u*s,v]$) is not full, then $[u,v]$ cannot be full
  either. Completing the proof, we conclude $\mu(u,v)=\mu(u*s,v*s)$.
\end{proof}
\end{lemma}

\begin{lemma}\label{le:Q-sum}
For all $v\in \iot$, 
\[
\sum_{u\le v}Q_{u,v}(q)=q^{\rho(v)}.
\]
\begin{proof}
We prove the lemma using induction on $\rho(v)$. Given $s\in
\DR(v)$, partition $I_v$ into three sets:
\[
\begin{split}
A & = \{u\le v\mid u* s < u\},\\
B & = \{u\le v\mid u* s > u\},\\
C & = \{u\le v\mid u* s = u\}.
\end{split}
\]
By the lifting property \cite[Lemma 2.7]{hultman}, the map $u\mapsto
u*s$ is a bijection between $A$ and $B$. The 
recursion in Proposition \ref{pr:Q-recursion} therefore yields
\[
\begin{split}
\sum_{u\le v}Q_{u,v}(q) &=  \sum_{u\in A}Q_{u* s,v* s}(q)\\
 & \,\,\,\, +\sum_{u\in
    B}(qQ_{u* s,v* s}(q)+(q-1)Q_{u,v* s}(q))\\
 & \,\,\,\, +\sum_{u\in
    C}qQ_{u,v* s}(q) \\
&= \sum_{\ontop{u\in A}{u \le v* s}}qQ_{u,v* s}(q)+ \sum_{\ontop{u\in
    B}{u\le v* s}}(1+q-1)Q_{u,v* s}(q)+\sum_{\ontop{u\in C}{u\le v* s}}
    qQ_{u,v* s}(q)\\ 
&= q\sum_{u\le v* s}Q_{u,v* s}(q),
\end{split}
\]
proving the claim. 
\end{proof}
\end{lemma}

\begin{lemma}\label{temp2}
We have $P_{u,v}(0) = 1$ whenever $u\le v$ in $\iot$.
\end{lemma}
\begin{proof}
The assertion is clear if $u=v$, and we employ induction on $\rho(v)-\rho(u)$.

Vogan's \cite[Corollary 6.12]{vogan} translates to
\[
q^{\rho(v)-\rho(u)}P_{u,v}(q^{-1}) = \sum_{u\le w \le v} Q_{u,w}(q)P_{w,v}(q).
\]
The left hand side is a polynomial with zero constant term. Hence, Lemma
\ref{le:Mobius} implies 
\[
P_{u,v}(0) = - \sum_{u< w \le v} \mu(u,w) = \mu(u,u) = 1,
\]
as desired.

\end{proof}
\end{lemma}

We are finally in position to prove the main results. Since all
necessary technical prerequisites have been established, the
corresponding arguments 
from \cite{CP} can now be transferred to our setting more or less verbatim.

\begin{theorem}\label{th:regtoratsmooth}
Suppose $u,v\in \iot$, $u\le w$. The following conditions are equivalent:
\begin{itemize}
\item[(i)] The degree of $v$ in $\BG(w)$ is $\rho(w)$ for all $u\le
  v\le w$.
\item[(ii)] The KLV polynomials satisfy $P_{v,w}(q)=1$ for all $u\le
  v\le w$. That is, the orbit closure $\COO{w}$ is rationally smooth at
  $\OO_u$.
\end{itemize}
\begin{proof}
Define
\[
f_{u,w}(q)=q^{\rho(w)-\rho(u)}(P_{u,w}(q^{-2})-1).
\]
The $P$-polynomials have nonnegative coefficients. By Lemma \ref{temp2},
$f_{u,w}(q)$ too is a polynomial with nonnegative
coefficients. Since it has vanishing constant term,
$f_{u,w}^\prime(1)=0$ if and only if $f_{u,w}(q)=0$ which, in turn, is
equivalent to $P_{u,w}(q)=1$.

Now,
\[
f_{u,w}^\prime(1) = (\rho(w)-\rho(u))(P_{u,w}(1)-1) - 2P^\prime_{u,w}(1).
\]
Since $Q_{u,w}(1)=\delta_{u,w}$, we have 
\[
\begin{split}
-2P^\prime_{u,w}(1) & =\frac{d}{dq} P_{u,w}(q^{-2})\arrowvert_{q=1} \\
& = 2(\rho(u)-\rho(w))P_{u,w}(1)+2\sum_{u\le v \le
 w}Q^\prime_{u,v}(1)P_{v,w}(1)+ 2P^\prime_{u,w}(1).
\end{split}
\]
Hence,
\[
f_{u,w}^\prime(1)= \rho(u)-\rho(w)+\sum_{u\le v\le w}Q^\prime_{u,v}(1)P_{v,w}(1).
\]

To begin with, assume (ii) holds. Then,
\[
\rho(w)-\rho(v) = \sum_{v \le v^\prime \le w}Q_{v,v^\prime}^\prime(1) 
\]
for all $u\le v \le w$. Condition (i) now follows from Lemma
\ref{le:edge} together with Lemma \ref{le:comparability}.

Finally, let us prove (i) $\Rightarrow$ (ii) by induction on
$\rho(w)-\rho(u)$. Suppose $u<v\le w$ in $\br(\iot)$. By Lemma
\ref{le:edge} and the induction assumption, 
$Q^\prime_{u,v}(1)P_{v,w}(1)$ is one if $\{u,v\}$ is an edge in
$\BG(w)$, zero otherwise. Since $\deg(u)=\rho(w)$, $u$ has exactly
$\rho(w)-\rho(u)$ neighbours $v$ such that $u<v$. We conclude
$f_{u,w}^\prime(1)=0$ as desired. 
\end{proof}
\end{theorem}

\begin{theorem}\label{th:rank_symmetry}
For $w\in \iot$, the following are equivalent:
\begin{itemize}
\item[(i)] The interval $[\id,w] = \br(I_w)$ has equally many elements of rank $i$ as of rank $\rho(w)-i$.
\item[(ii)] The graph $\BG(w)$ is regular.
\item[(iii)] $P_{u,w}(q)=1$ for all $u\leq w$.
\end{itemize}
\end{theorem}
\begin{proof}
(i) $\Rightarrow$ (ii): Let $n(i)$ denote the number of
  elements of rank $i$ in $[e,w]$. Now, using Lemma
  \ref{le:comparability} and Theorem \ref{th:inequality}, we count the
  edges in $\BG(w)$ in two ways and obtain
\[
\sum_{i=0}^{\rho(w)}n(i) i \ge \sum_{i=0}^{\rho(w)} n(i)(\rho(w)-i)
\]
with equality if and only if $\BG(w)$ is $\rho(w)$-regular. However,
if $n(i)=n(\rho(w)-i)$ for all $i$, then equality does hold.

(ii) $\Rightarrow$ (iii): This follows from Theorem
\ref{th:regtoratsmooth}. 

(iii) $\Rightarrow$ (i): We claim that
\[
F_w(q) = \sum_{u\leq w}P_{u,w}(q)q^{\rho(u)}
\]
is a symmetric polynomial, i.e.\ $F_w(q) =
q^{\rho(w)}F_w(q^{-1})$. If the $P$-polynomials are all $1$, this means 
\[
\sum_{u\le w}q^{\rho(u)} = \sum_{u\le w}q^{\rho(w)-\rho(u)}.
\]
It therefore remains to verify the claim. Observe that
\[
\begin{split}
q^{\rho(w)}F_w(q^{-1}) & = \sum_{u\le
  w}q^{\rho(w)-\rho(u)}P_{u,w}(q^{-1})\\
& =\sum_{u\le w}\sum_{u\le v\le
  w}Q_{u,v}(q)P_{v,w}(q) \\
& = \sum_{v\le w}P_{v,w}(q)\sum_{u\le v}Q_{u,v}(q)\\
& = F_w(q),
\end{split}
\]
where the last equality follows from Lemma \ref{le:Q-sum}.
\end{proof}

To illustrate these results, consider Figure \ref{fi:bg}. The interval
$[\id,w]$ has three elements of rank three but only two of rank
$\rho(w)-3=1$. By Theorem \ref{th:rank_symmetry}, $\COO{w}$ is
rationally singular. A more careful inspection of the graph shows that
$s_5s_1$ and $e$ both have degree five whereas all other vertices have
degree $\rho(w)=4$. By Theorem \ref{th:regtoratsmooth}, the rationally
singular locus of $\COO{w}$ therefore is
$\OO_{s_5s_1}\cup\OO_e$. Also, observe that the degree never
decreases as we move down in the graph. This phenomenon is explained
in the next section.

\section{Sufficiency of the bottom vertex}\label{se:bottom}
In this final section, the criterion given in Theorem
\ref{th:regtoratsmooth} is significantly improved in the special case
$G=\SL_{2n}(\C)$, $K=\Sp_{2n}(\C)$. In that case, as we shall see,
whether or not an orbit closure $\COO{w}$ is rationally smooth at
$\OO_u$ is determined by the degree of $u$ alone (Corollary
\ref{co:bottom} below). The corresponding
statement for Schubert varieties is known to be true in type $A$
\cite{deodhar} but false in general (see \cite{BG} for some
elaboration on this). Necessarily, therefore, our arguments must be
type specific since they cannot possibly be extended to the situation
in Example \ref{ex:Schubert} for arbitrary $G^\prime$. 

\subsection{Notation and preliminaries}
Let us spend a few lines fixing notation with respect to the
combinatorics of symmetric groups. 

We work in the set $\fix$ of fixed point free involutions on
$\{1, \dots, 2n\}$. Let $\star$ denote the conjugation action
from the right 
by the symmetric group $S_{2n}$ on itself, i.e.\ $\sigma\star \pi =
\pi^{-1} \sigma \pi$. Then, $\fix = w_0\star S_{2n}$, where $w_0$ is the
reverse permutation $i\mapsto 2n+1-i$.

If $t=(a,b)\in S_{2n}$ is a transposition and $u\in \fix$, then
$u\star t= u$ if and only if $t$ is a $2$-cycle in the cycle decomposition of
$u$. If $u\star t\neq u$, the decompositions into $2$-cycles of
$u$ and $u\star t$ are as follows:
\[
\begin{split}
u &= (a, u(a))(b, u(b))\cdots,\\
u\star t &= (a, u(b))(b, u(a))\cdots,
\end{split}
\] 
where the dots denote the remaining $2$-cycles (that both involutions
have in common). In particular, there is exactly one transposition
$t^\prime\neq t$ such that $u\star t^\prime = u\star t$, namely
$t^\prime=(u(a), u(b))$.

Let $\preceq$ denote the dual of the subposet of the Bruhat order on $S_{2n}$ induced by
$\fix$. The bottom element of this poset is $w_0$. Observe that if
$u\neq u\star t$, then $u\star t \succ u$ iff $t$ is an {\em inversion} of $u$
(meaning $t=(a,b)$ with $a<b$ and $u(a)>u(b)$). If $s =(i,i+1)$ is an
adjacent transposition, then $s$ 
is a {\em descent} if it is an inversion; otherwise $s$ is an
{\em ascent}.


For $u\in \fix$, $1\le i,j\le 2n$, define 
\[
u_{(i,j)} = |\{x\in \{1, \dots, 2n\} \mid x\le i\text{ and } u(i)\ge j\}|.
\]
Thus, $u_{(i,j)}$ is the number of dots weakly northwest of $(i,j)$ in
the permutation diagram of $u$. 
\begin{lemma}[Standard criterion; Theorem 2.1.5 in \cite{BB}]
For $u,w\in \fix$, we have $u\preceq w$ iff $u_{(i,j)}\ge w_{(i,j)}$ for
all $(i,j)$.
\end{lemma}

For $w\in \fix$,
define the {\em Bruhat graph} $\bg(w)$ as the graph whose vertex set
is $I_w = \{u\in \fix\mid u\preceq w\}$ and $\{u,v\}$ is an edge iff
$u\neq v=u\star t$ for some transposition $t$. Thus, each edge has
exactly two transpositions associated with it, and the graph is simple
(no loops or multiple edges). If $w$ is understood from the context
and $u\preceq w$, let $\out(u)$ denote the set of edges incident to $u$ in
$\bg(w)$. Also, define $\deg(u)=|\out(u)|$.

\begin{proposition}\label{pr:notation}
Suppose $W=A_{2n-1}\cong S_{2n}$ with $\theta:W\to W$ given by the
unique nontrivial involution of the Dynkin diagram. Then, $x\mapsto
w_0x$ defines a bijection $\fix \to 
\iot$. Moreover, the bijection is an isomorphism of Bruhat graphs,
i.e.\ $u\preceq w \Leftrightarrow w_0u\le w_0w$ and $w_0(w\star t) =
w_0w* t$.
\begin{proof}
This is immediate from the well known facts that
$\theta(x)=w_0xw_0$ and that $x\mapsto w_0x$ is an antiautomorphism of
$\br(W)$. 
\end{proof}
\end{proposition}

\subsection{An injective map}

Suppose $w\succeq u\neq w_0$ and let $r=(i,j)$, $i<j$, be a
transposition such that $u\star r \prec u$. Let $a=u(i)$ and
$b=u(j)$. Thus, $a<b\neq i$.

For a transposition $t=(x,y)$, we use the notation $\supp(t)=\{x,y\}$.

\begin{definition} A transposition $t$ is called {\em compatible}
  (with respect to $u$ and $r$) if either $\supp(t)\cap \{a,b,i,j\} =
  \emptyset$ or $\supp(t)\cap\{i,j\} \neq \emptyset$.
\end{definition}

Given an edge $e\in \out(u)$ there are precisely two
transpositions $t$ and $t^\prime\neq t$ such that $e=\{u,u\star
t\}=\{u,u\star t^\prime\}$. At least one of them is
compatible.

\begin{definition} For any edge $e\in \out(u)$, let $t_e$ be a compatible transposition such that $e=\{u,u\star t_e\}$. 
\end{definition}

\begin{definition} \label{de:epsilon}
Given $e\in \out(u)$, define $\epsilon(e) = \{u\star r, u\star
r \tau_e\}$, where 
\[
\tau_e = 
\begin{cases}
rt_er & \text{if $u\star t_er\preceq w$,}\\
t_e & \text{otherwise.}
\end{cases}
\] 
\end{definition}

The point of all this is the following:

\begin{theorem}\label{th:main}
Definition \ref{de:epsilon} defines an injective map
$\epsilon:\out(u)\to \out(u\star r)$.
\begin{proof}
This follows from Lemmata \ref{le:well-def}, \ref{le:map} and
\ref{le:injective} below. 
\end{proof}
\end{theorem}

By Theorem \ref{th:main}, the degree can never decrease as we go down
along edges in a Bruhat graph. In particular, if a vertex has the
minimum possible degree, then so does every vertex above it:

\begin{corollary}\label{co:bottom}
We have $\deg(v)=\deg(w)$ for all $u\preceq v\preceq
w$ if and only if $\deg(u)=\deg(w)$. 
\end{corollary}

Thus, to determine whether condition (i) of Theorem
\ref{th:regtoratsmooth} is satisfied, it suffices to check the degree
of $u$.

\begin{remark}{\em 
The set $\F = \{w\in F_{2n} \mid i\le n \Rightarrow w(i) \ge n+1\}$ is
in natural bijective correspondence with $S_{2n}$ in a way which
identifies $\br(S_{2n})$ with $\preceq$. Restricted to
$w\in \F$, Corollary 
\ref{co:bottom} specialises to a result of Deodhar \cite{deodhar} for
type $A$ Schubert varieties. In that setting, our arguments are closely
related to work of Billey and Warrington \cite[\S 6]{BW}
}\end{remark}

\begin{remark}{\em 
Observe that for $G=\SL_{2n}(\C)$, $K=\Sp_{2n}(\C)$, Theorem
\ref{th:inequality} follows directly from Theorem \ref{th:main}. Thus, we have
reproven Brion's \cite[Theorem 2.5]{brion} in this case. 
}\end{remark}

\subsection{Proof of Theorem \ref{th:main}}

\begin{lemma}\label{le:well-def}
The set $\epsilon(e)$ is well defined, i.e.\ independent of the choice of
$t_e$.
\begin{proof}
This is clear if $\supp(t)\cap \{a,b,i,j\}=\emptyset$. If not, the only case
when both transpositions associated with $e$ are compatible 
is when $e = \{u,u\star t\}=\{u,u\star t^\prime\}$ for
$\{t,t^\prime\}=\{(i,b), (j, a)\}$. In this case, we have $u\star tr
= u\star t^\prime r$ and $u\star rt = u\star r = u\star rt^\prime$. 
\end{proof} 
\end{lemma}

\begin{lemma}\label{le:map}
For every $e\in \out(u)$, we have $\epsilon(e)\in \out(u\star r)$.
\begin{proof}
We must show that $u\star r \neq u\star r\tau_e \preceq w$.

First, assume $u\star r = u\star r\tau_e$. Then, $u\star t_e = u\star
r\tau_ert_e$. If $\tau_e = rt_er$, this means $u\star t_e = u$ which
contradicts the fact that $e\in \out(u)$. If, on the other hand,
$\tau_e=t_e$, then we conclude that $r$ and $t_e$ do not commute,
hence that $rt_ert_er = t_e$. But then, $u\star t_e r = u\star rt_ert_er
= u\star t_e\preceq w$ which contradicts $\tau_e=t_e$. Thus, $u\star
r \neq u \star r\tau_e$.

It remains to prove $u\star r\tau_e \preceq w$, i.e.\ that either $u\star
t_er\preceq w$ or $u\star rt_e\preceq w$ (or both). There are a few cases:\\


{\sf Case 1.} If $\supp(t_e)\cap\{i,j,a,b\}=\emptyset$, then $u\star
t_e(i)<u\star t_e(j)$. Thus, $w\succeq u\star t_e\succ u\star t_e r$.\\ 

{\sf Case 2.} If $t_e = (i,j)$, then $u\star t_e r = u\preceq w$.\\

{\sf Case 3.} If $t_e = (i,b)$, then $u\star t_e(i)=j$ so that
$u\star t_er = u\star t_e \preceq w$.\\

{\sf Case 4.} If $t_e= (j,a)$, we again have $u\star
t_e(i)=j$.\\

{\sf Case 5.} If $t_e = (i,k)$ with $k\not \in \{j,a,b\}$,
then $rt_er = t_ert_e = (j, k)$ and $u\star rt_e = u\star t_ert_er$. Let $c =
u(k)$. 

We have $u\star t_e(i)=c$, $u\star t_e(j) =
b$, $u\star t_e(k)=a$, $u\star r(i)=b$ and $u\star r(k)=c$. If $k<j$,
it follows that $w\succeq u\star 
t_e \succ u \star t_ert_er = u\star rt_e$. Otherwise, $k>i$ and either
$w\succeq 
u\star r\succ u\star rt_e$ (if $b<c$) or $u\star t_er\prec u\star t_e$
(if $b>c$).\\  


{\sf Case 6.} If $t_e = (j,k)$ with $k\not \in \{i,a,b\}$,
then $rt_er = t_ert_e = (i, k)$ and $u\star rt_e = u\star
t_ert_er$. Again, let $c = u(k)$.

Now, $u\star t_e(i) = a$, $u\star t_e(j) = c$ and $u\star t_e(k) =
b$ so that $u\star t_er \succ u\star t_e$ implies $c<a<b$. It follows that
either $u\star rt_e = u\star t_ert_er \prec u\star t_e \preceq w$ (if $i<k$)
or $u\star r\succ u\star rt_e$ (if $k<j$).
\end{proof}
\end{lemma}

\begin{lemma}\label{le:injective}
If  $e\neq e^\prime$ for $e, e^\prime \in \out(u)$, then
$\epsilon(e)\neq \epsilon(e^\prime)$.
\begin{proof}
Suppose $e, e^\prime \in \out(u)$ and $e\neq e^\prime$. There are
three cases:\\

{\sf Case 1.} If $\tau_e= rt_er$ and $\tau_{e^\prime}=rt_{e^\prime}r$,
then $\epsilon(e) = \epsilon(e^\prime) \Leftrightarrow u\star t_er =
u\star t_{e^\prime}r \Leftrightarrow u\star t_e = u\star t_{e^\prime}$
which contradicts $e \neq e^\prime$.\\

{\sf Case 2.} Suppose $\tau_e= t_e$ and
$\tau_{e^\prime}=t_{e^\prime}$. Assume $\epsilon(e) =
\epsilon(e^\prime)$, i.e.\ $u\star rt_e = 
u\star rt_{e^\prime}$. If both $t_e$ and $t_{e^\prime}$ commute with
$r$ we argue as in the previous case. If not, since both $t_e$ and
$t_{e^\prime}$ are compatible, we either have $\{t_e,t_{e^\prime}\} = \{(i,b),(j, a)\}$ leading to the contradiction
$u\star t_e = u\star t_{e^\prime}$, or we have $\{t_e,
t_{e^\prime}\} = \{(i,a), (j, b)\}$ which implies the contradiction
$u=u\star t_e$.\\

{\sf Case 3.} Finally, assume $\tau_e= rt_er$ and
$\tau_{e^\prime}=t_{e^\prime}$. Then, the assumption $\epsilon(e) =
\epsilon(e^\prime)$ amounts to $u\star t_e = 
u\star rt_{e^\prime}r$. Suppose $rt_{e^\prime}\neq t_{e^\prime}r$ and
$rt_e\neq t_er$; otherwise we would be in Case 1 or 2,
respectively. This implies that either $rt_{e^\prime}r =
t_e$ or $\{t_e, rt_{e^\prime} r\} = \{(i,b),(j,a)\}$. The latter
case, though, leads to $u\star r = u\star
rrt_{e^\prime}r = u\star t_{e^\prime}r$ implying the contradiction $u
= u\star t_{e^\prime}$. Thus, $rt_{e^\prime}r =
t_e$, and therefore $\{t_e, t_{e^\prime}\} = \{(i,k),(j,k)\}$ for some $k\not
\in \{i,j,a,b\}$. Let us suppose $t_e=(i,k)$; the other case is
completely similar. A small computation shows that 
\[
\begin{split}
u &=           (i,a)(j,b)(k,c)\cdots,\\
u\star r &=    (i,b)(j,a)(k,c)\cdots,\\
u\star t_e &=  (i,c)(j,b)(k,a)\cdots,\\
u\star rt_e &= (i,c)(j,a)(k,b)\cdots,\\
u\star t_er &= (i,b)(j,c)(k,a)\cdots,\\
u\star rt_er &= (i,a)(j,c)(k,b)\cdots, 
\end{split}
\] 
where we have written down the $2$-cycle decompositions of the various
elements (the dots indicate the remaining cycles; they are equal for
all six elements). In particular, the elements are all distinct, so
$|u\star \langle r, t_e\rangle|=6$. 

Now observe that precisely five of the elements in $u\star \langle
r,t_e\rangle$ belong to $\bg(w)$; the one which does not is $u\star
rt_e=u\star t_{e^\prime}r$, because $\tau_{e^\prime}\neq
rt_{e^\prime}r$. A contradiction is now provided by Lemma
\ref{le:top_element} below.
\end{proof}
\end{lemma}

\begin{lemma} \label{le:top_element}
Suppose $|u\star \langle t_1, t_2\rangle\cap I_w|\ge 5$ for two
elements $u,w\in \fix$ and some
transpositions $t_1, t_2$. Then, $|u\star \langle t_1, t_2\rangle\cap
I_w| = 6$, i.e.\ $u\star \langle t_1, t_2\rangle \subseteq I_w$.
\begin{proof}
The set of transpositions in the dihedral subgroup $\langle
t_1,t_2\rangle$ is $\{t_1, t_2, t_1t_2t_1\} = \{(x_1,
x_2),(x_2,x_3),(x_1,x_3)\}$ for some $1\le x_1<x_2<x_3\le 2n$. There
are elements $1\le a_1<a_2<a_3\le 2n$, with $x_i\neq a_j$ for all $i$
and $j$, such that $u\star \langle t_1, t_2\rangle$ consists of the
six involutions with cycle decomposition of the form
\[
(x_1, a_{i_1})(x_2, a_{i_2})(x_3, a_{i_3})\cdots,
\]
dots denoting the $2$-cycles in $u$ with support disjoint
from $\{x_1, x_2, x_3, a_1, a_2, a_3\}$.

In order to simplify notation, let
\[
[i_1i_2i_3] = (x_1, a_{i_1})(x_2, a_{i_2})(x_3, a_{i_3})\cdots.
\]
Since $[123]$ is the maximum element in $u\star \langle t_1,
t_2\rangle$, it suffices to show that $w\succeq [123]$ whenever $w\succeq
[213]$ and $w\succeq [132]$. To this end, consider the standard
criterion. For $1\le \alpha,\beta\le 2n$, let
\[
\begin{split}
[i_1i_2i_3]_{(\alpha,\beta)}^+ &= |\{j\in \{1,2,3\}\mid x_j\le
\alpha\text{ and }a_{i_j}\ge \beta \},\\
[i_1i_2i_3]_{(\alpha,\beta)}^- &= |\{j\in \{1,2,3\}\mid a_{i_j}\le \alpha\text{ and
}x_i\ge \beta \}.
\end{split}
\] 
Then, the number of dots weakly northwest of $(\alpha, \beta)$ in the
diagram of $[i_1i_2i_3]$ is
\[
[i_1i_2i_3]_{(\alpha, \beta)}=[i_1i_2i_3]_{(\alpha,\beta)}^+ +
[i_1i_2i_3]_{(\alpha,\beta)}^- + D,
\]
where $D$ counts dots with coordinates outside
$\{x_1,x_2,x_3,a_1,a_2,a_3\}$; this number is independent of $i_1,
i_2, i_3$.

By the symmetry between $x$ and $a$, it is sufficient to show
\[
[123]_{(\alpha,\beta)}^+ = \min \left([213]_{(\alpha,\beta)}^+,[132]_{(\alpha,\beta)}^+\right)
\]
for all $\alpha, \beta$. This statement, however, follows immediately from
the observation that for all $m$, the first $m$ letters in the string
``$123$'' are the same as the first $m$ letters in one of the strings
``$213$'' and ``$132$''.
\end{proof}
\end{lemma}

\end{document}